\documentclass[12pt]{amsart}
\usepackage{amsmath,amsfonts,amssymb,amscd}
\usepackage[all]{xy}
\usepackage{graphics}
\usepackage{epsfig}
\usepackage{psfrag}

\def\ignore #1 {}

\newtheorem{lemma}{Lemma}
\newtheorem{thm}{Theorem}
\newtheorem{prop}{Proposition}
\newtheorem{proposition}{Proposition}

\newtheorem{corollary}{Corollary}

\newtheorem{remark}{Remark}

\newtheorem{ex}{Example}

\newtheorem*{ack}{Acknowledgments}

\begin{document}

\title{Meromorphic Line Bundles and Holomorphic Gerbes}
\author{Edoardo Ballico and Oren Ben-Bassat}
 \maketitle

\centerline{\bf }


{\tiny \tableofcontents }




\begin{abstract} We show that any complex manifold that has a divisor whose normalization has non-zero first Betti number either has a non-trivial holomorphic gerbe which does not trivialize meromorphicly or a meromorphic line bundle not equivalent to any holomorphic line bundle.  Similarly, higher Betti numbers of divisors correspond to higher gerbes or meromorphic gerbes.  We give several new examples.
\end{abstract}
\section{Introduction}\label{intro}  
Let $X$ be a complex manifold.  Unless stated otherwise, we always work in the classical analytic topology.   Let $\mathcal{O}$ denote the sheaf of holomorphic functions.   Let $\mathcal{O}^{\times}$ denote the nowhere vanishing holomorphic functions and $\mathcal{M}^{\times}$ the non-zero meromorphic functions Consider the short exact sequence of sheaves of groups on $X$  
\begin{equation}\label{eqn:HolMer}1 \to \mathcal{O}^{\times} \to \mathcal{M}^{\times} \to \mathcal{M}^{\times}/\mathcal{O}^{\times} \to 1
\end{equation}
and its induced short exact sequence of cohomology groups 
\small
\begin{equation}\label{eqn:decom}
1 \to H^{p}(X, \mathcal{M}^{\times})/H^{p}(X, \mathcal{O}^{\times}) \to H^{p}(X,\mathcal{M}^{\times}/\mathcal{O}^{\times}) \to \text{ker}[H^{p+1}(X,\mathcal{O}^{\times}) \to H^{p}(X,\mathcal{M}^{\times})] \to 1.
\end{equation}
\normalsize
We use $Div(X)=H^{0}(\mathcal{M}^{\times}/\mathcal{O}^{\times})$ to denote the group of divisors on $X$ and $IDiv(X)$ to denote the set of irreducible divisors.
We will relate the topology of divisors inside $X$ to the middle term by producing an injective map 
\[B: H^{p}(\tilde{D},\mathbb{Z}) \to  H^{p}(X,\mathcal{M}^{\times}/\mathcal{O}^{\times})
\]
from the cohomology of the normalization $\tilde{D}$ of any divisor $D$ in $X$.   In the case $p=1$ one can then consider the resulting classes in $H^{2}(X,\mathcal{O}^{\times})$ which are represented by holomorphic gerbes or if such a class is trivial, one produces classes represented by meromorphic line bundles.   Most of our examples then use techniques of classical algebraic geometry to produce divisors whose normalization has interesting topology.  We comment on the analytic and topological characteristics of the resulting gerbes.
\begin{ack}
The first author was partially supported by MIUR and GNSAGA of INdAM (Italy).  The second author would like to thank the University of Haifa for travel support and Marco Andreatta, Fabrizio Catanese,  Elizabeth Gasparim, the University of Trento and Fondazione Bruno Kessler for supporting his stay in Trento.
\end{ack}
\subsection{Comparison to other works}
Before this note was finished we found the article \cite{ChKeLe10} of 
X. Chen, M. Kerr, and J. Lewis  that produced many examples of non-trivial meromorphic line bundles on smooth projective complex algebraic varieties.  This contradicts the sometimes heard claim that meromorphic line bundles on smooth projective complex algebraic varieties are trivial.  A similar yet more general construction was realized independently by the second author, so we will give here a rough comparison.  Let $X$ be a complex manifold.  Suppose that we have a non-constant meromorphic function $f$ on $X$ whose divisor $div(f)=D =\sum m_{i} D_{i}$ is supported on a normal subvariety $\cup_{\{i | m_{i} \neq 0\}} D_{i}$ of $X$.  Let $n \in H^{1}(X,\mathbb{Z})$ be some cohomology class which restricts to a non-zero element in at least one of the $H^{1}(D_{i},\mathbb{Z})$ represented in $div(f)$ by a non-zero coefficient $m_{i}$.  Consider the image of $f \otimes n$ under the cup product
\[\mathcal{M}^{\times}(X) \otimes H^{1}(X,\mathbb{Z}) \to H^{1}(X,\mathcal{M}^{\times}).
\]
Now its easy to see that the divisor along $D$ of the image of our element is just $\sum m_{i}n|_{D_{i}}$ and so by our assumption on the class of $n$, the class in $H^{1}(X,\mathcal{M}^{\times})$ is non-trivial.  This kind of construction appeared in  \cite{ChKeLe10} where it was shown that in general, $H^{i}(X,\mathcal{M}^{\times})$ do not vanish when $X$ is smooth and projective.  Upon reading about this construction, one may wonder if the divisors can be used directly, not just to measure non-triviality, but to produce themselves meromorphic line bundles.  This is our approach.  Such a philosophy is inherent in Brylinski's work \cite{Br1994} in which he outlined a framework for the categorification of Beilinson's regulator maps from algebraic K-theory to Deligne cohomology.  Both  \cite{ChKeLe10}  and \cite{Br1994} are expressed in terms of sheaves and also cover higher cohomology classes.  In Brylinski's case, he would produce starting from some integral cohomology classes on divisors a pair consisting of a holomorphic n-gerbe and a meromorphic trivialization.   When this holomorphic n-gerbe is trivial, a trivialization of it is a meromorphic (n-1)gerbe.  So for example, a meromorphic line bundle is a meromorphic trivialization of a trivial holomorphic gerbe.  As a consequence of the current paper we produce new examples of complex manifolds admitting non-trivial meromorphic line bundles, perhaps the most surprising being the affine plane $\mathbb{C}^{2}$   and the projective plane $\mathbb {P}^2$ (see subsection \ref{c2}).   Constructions of this type contradict statements sometimes seen in the literature such as an incidental remark in line 2 of the proof of Theorem 9 of \cite{Go1987} concerning the absence of non-trivial meromorphic line bundles on smooth projective varieties.  Other examples of such statements are given in \cite{ChKeLe10}.  We answer the open question of \cite{ChKeLe10} in subsection \ref{Question}.  Some material on material on holomorphic gerbes both old and recent can be found in \cite{Chatterjee}, \cite{Br1993}, \cite{HL}, \cite{Hao}, \cite{Be2009} and \cite{Hi2010}.

\section{The Main Technical Tools}\label{mero}

\subsection{The Splitting Construction and the map $B$}
The construction in this section was inspired by Brylinski's paper \cite{Br1994} where he give an algebro-geometric (or complex analytic) construction of holomorphic gerbes analogous to the construction of a line bundles from a divisor.   We produce a splitting of one of the maps used in his construction which we call $div$ below.
Let $D$ be a divisor on $X$, $Y$ the support of $D$ and $\nu_{Y}$ the normalization $\tilde{Y}$ of $Y$ followed by the inclusion of $Y$ into $X$.
\[\nu_{Y}: \tilde{Y} \to X
\]
We have a surjective map of sheaves of groups
\[div: \mathcal{M}^{\times}/\mathcal{O}^{\times}  \to \nu_{Y*}\mathbb{Z}_{\tilde{Y}}
\]
\[f \mapsto [C \mapsto div(f, \nu_{Y}(C))].
\]
This means the following: for any open set $U$ and each component $C$ of $\nu_{Y}^{-1}(U \cap Y)$ we assign the coefficient of $\nu_{Y}(C)$ in $div(f) \cap U$. 
This map has a splitting as sheaves of groups.

In order to define the splitting, we first need to choose an open cover $\{U_{i} \}$ with the property that all irreducible components $Y^{(\alpha)}_{i}$ of $U_{i} \cap Y$ are locally irreducible and there exist holomorphic functions $f^{(\alpha)}_{i}$ on $U_{i}$ whose divisor is $Y^{(\alpha)}_{i}$.    The fact that such a cover exists follows from the noetherian property of complex analytic spaces: in
$\mathcal{O}_{Y,y}$ the ideal $(0)$ has a primary decomposition and each irreducible component at $y$ (there are finitely many) stays irreducible in a small neighborhood of $y$.    See Proposition 11 page 55--56 and Corollary 1 page 68 of \cite{Na1966} for these facts.

Note that on $U_{i}$ we can identify $\nu_{Y*}\mathbb{Z}$ with $\prod_{\alpha}j^{(\alpha)}_{*}\mathbb{Z}$ where the maps $j^{(\alpha)}$ are the restrictions of $\nu_{Y}$ to the component  $Y^{(\alpha)}_{i}$:
\[j^{(\alpha)}: Y^{(\alpha)}_{i} \to X.
\]
We now define maps of sheaves of groups 
\begin{equation}\label{eqn:Splitting_i}
s_{i}: (\nu_{Y*}\mathbb{Z}_{\tilde{Y}})|_{U_{i}} \to \mathcal{M}^{\times}|_{U_{i}}
\end{equation}
\[
n_{i}^{(\alpha)} \mapsto \prod_{\alpha} (f^{(\alpha)}_{i})^{n_{i}^{(\alpha)}}.  
\]

To see that these maps glue together to a map of sheaves
\begin{equation}\label{eqn:Splitting}
s: \nu_{Y*}\mathbb{Z}_{\tilde{Y}} \to \mathcal{M}^{\times}/\mathcal{O}^{\times}
\end{equation}
 we evaluate on $U_{i} \cap U_{j}$ the function 
\[(s_{i}(n_{i}^{(\alpha)}))(s_{j}(n_{j}^{(\beta)}))^{-1} = (\prod_{\alpha} (f^{(\alpha)}_{i})^{n_{i}^{(\alpha)}})(\prod_{\beta} (f^{(\beta)}_{j})^{n_{j}^{(\beta)}})^{-1}
\]
in the case where $n_{i}^{(\alpha)}$ and $n_{j}^{(\beta)}$ agree on the intersection.  We need to show that this function is holomorphic and nowhere vanishing.
For any component of $Y \cap U_{i} \cap U_{j}$ appearing in the intersection  $U_{i} \cap U_{j}$ there are unique corresponding components $Y^{(\alpha)}_{i}$ and $Y^{(\beta)}_{j}$ and the function $f^{(\alpha)}_{i}$ can be written as $f^{(\beta)}_{j}$ multiplied by a holomorphic function.  Also, for any such $\gamma$,  $n_{i}^{(\alpha)}$ and $n_{j}^{(\beta)}$ are the same numbers $n^{(\gamma)}_{i,j}$.  On the other hand, for any $\alpha$ or $\beta$ not corresponding to a component of the intersection, the functions $f^{(\alpha)}_{i}$ and $f^{(\beta)}_{j}$ are already holomorphic and nowhere vanishing.  It is also easy to see that the definition of $s$ is independent of the choices made in the construction.

\begin{lemma}\label{splitting}
This induces a splitting of the previous map:
\[div \circ s = \text{id}_{\nu_{Y*}\mathbb{Z}_{\tilde{Y}}}.
\]
As a result, for any complex manifold $X$ and any divisor $D$ in $X$ with support $Y$ there is a subgroup 
\[H^{p}(\tilde{Y},\mathbb{Z}) \subset H^{p}(X,\mathcal{M}^{\times}/\mathcal{O}^{\times} )
\]
where $\tilde{Y}$ is the normalization of $Y$.
\end{lemma}

{\bf Proof.}
In order to see that it is a splitting, it is enough to show it locally after restriction to each $U_{i}$. The components 
of $\nu_{Y}^{-1}(Y \cap U_{i})$ are precisely the $\nu_{Y}^{-1}(Y^{(\alpha)}_{i})$ and hence $\nu_{Y}(\nu_{Y}^{-1}(Y^{(\alpha)}_{i})) = Y^{(\alpha)}_{i}$.  The divisor of $\prod_{\alpha} (f^{(\alpha)}_{i})^{n^{(\alpha)}_{i}}$ along $Y^{(\alpha)}_{i}$ is $n^{(\alpha)}_{i}$.

Therefore 
\begin{equation}\label{CEqn} \mathcal{M}^{\times}/\mathcal{O}^{\times} = \nu_{Y*}\mathbb{Z}_{\tilde{Y}} \oplus \mathcal{C}_{Y}
\end{equation}
for some compliment sheaf of abelian groups $\mathcal{C}_{Y}$.
The normalization map $\nu_{Y}$ is a proper holomorphic map and its fibers are finite sets (\cite{Na1966}, part (a) of Definition 2 at page 114 and Theorem 4 at page 118, or \cite{f}, Appendix to Chapter 2). Furthermore Theorem  4.1.5 (i) (b) of \cite{D2004} says (as is true for the pushforward of a constructible sheaf under any proper analytic map) that the sheaf $\nu_{Y*}\mathbb{Z}_{\tilde{Y}}$ is constructible with respect to some analytic Whitney stratification of $X$.  Therefore by Theorem 4.1.9 of  \cite{D2004} $X$ has a cover by open sets $U$ such that 
\[H^{0}(U, R^{q}\nu_{Y*}\mathbb{Z}_{\tilde{Y}}) = {(R^{q}\nu_{Y*}\mathbb{Z}_{\tilde{Y}})}_{u}.
\]
for some point $u$ of $U$.
Finally by Theorem 2.3.26 of page 41 of \cite{D2004} we have
\[{(R^{q}\nu_{Y*}\mathbb{Z}_{\tilde{Y}})}_{u} = H^{q}(\nu_{Y}^{-1}(u),\mathbb{Z})
\]
which vanishes for $q>0$.  Therefore the Leray-Serre spectral sequence for $\nu_{Y}$ gives 

\[H^{p}(\tilde{Y},\mathbb{Z}) = H^{p}(X,\nu_{\tilde{Y}*} \mathbb{Z}) \subset H^{p}(X,\mathcal{M}^{\times}/\mathcal{O}^{\times} ).
\]

 \hfill $\Box$
 
Denote by $B$ the resulting group homomorphism 
\begin{equation}\label{eqn:Bdef}
B: \bigoplus_{D \in IDiv(X)} H^{p}(\tilde{D},\mathbb{Z}) \to  H^{p}(X,\mathcal{M}^{\times}/\mathcal{O}^{\times} ).
\end{equation}

\begin{proposition}\label{ee}
Let $X$ be an $n$-dimensional smooth and connected projective variety, $n \ge 1$. Fix an integer $t\ge 1$. Then $H^{2n-2}(X,\mathcal {M}^{\times}/\mathcal {O}^{\times })$ contains a subgroup isomorphic
to $\mathbb {Z}^t$.
\end{proposition}

{\bf Proof.}
Fix $t$ distinct smooth and connected hypersurfaces $D_1,\dots ,D_t$ of $X$. Set 
\[Y= D_1\cup \cdots \cup D_t.\] Thus the normalization $\widetilde{Y}$ of $Y$ is the disjoint union
$D_1\sqcup \cdots \sqcup D_t$. Thus $H^{2n-2}(\widetilde{Y},\mathbb {Z}) \cong \mathbb {Z}^t$. Apply Lemma \ref{splitting}.

 \hfill $\Box$

\subsection{Algebraic Structure}
Suppose now that $X=X_{Zar}(\mathbb{C})$ where $X_{Zar}$ is some regular scheme of finite type over $\mathbb{C}$.  Consider the inclusion map 
\[j:X \to X_{Zar}
\]
We have the following short exact sequences of sheaves and groups together with morphisms between them
\[
\xymatrix{\ar @{} [dr] |{}  & 1\ar[d] & 1\ar[d] & 1\ar[d] & &\\
1 \ar[r] & j^{-1}\mathcal{O}_{alg}^{\times} \ar[r] \ar[d]&  j^{-1}\mathcal{M}_{alg}^{\times}  \ar[r] \ar[d]&  j^{-1}(\mathcal{M}_{alg}^{\times}/\mathcal{O}_{alg}^{\times}) \ar[r] \ar[d] & 1\\
1 \ar[r] & \mathcal{O}^{\times} \ar[d] \ar[r] & \mathcal{M}^{\times} \ar[d] \ar[r] & \mathcal{M}^{\times}/\mathcal{O}^{\times} \ar[d] \ar[r] & 1 \\
1 \ar[r] & \mathcal{O}^{\times}/  j^{-1}\mathcal{O}_{alg}^{\times} \ar[r] \ar[d] & \mathcal{M}^{\times}/ j^{-1}\mathcal{M}_{alg}^{\times} \ar[r] \ar[d] & (\mathcal{M}^{\times}/\mathcal{O}^{\times} )/(j^{-1}(\mathcal{M}_{alg}^{\times}/\mathcal{O}_{alg}^{\times}))\ar[r] \ar[d]& 1 &
\\
& 1& 1 & 1. & &
}
\]
The sheaf $\mathcal{M}_{alg}^{\times}$ is a constant sheaf of non-zero rational functions and the pullback is the constant sheaf 
\[j^{-1}\mathcal{M}_{X, alg}^{\times} = \mathbb{C}(X)^{\times}
\]
in the analytic topology.
\begin{lemma}Let $X$ be a smooth projective algebraic variety considered as a complex manifold.  Let $f$ be a non-zero rational function on $X$ such that the support of every irreducible divisor appearing in $div(f)$ is normal and locally irreducible.   Then the diagram 
\[
\xymatrix{\ar @{} [dr] |{}    & &\\
H^{p}(X,\mathbb{Z}) 
\ar[r] \ar[d]&  H^{p}(X,\mathbb{C}(X)^{\times})  \ar[d] \\
\bigoplus_{D \in Div(X)}H^{p}(D,\mathbb{Z}) \ar[r] & H^{p}(X,\mathcal{M}^{\times}/\mathcal{O}^{\times})  }
\]
commutes.  The top map is the cup product with $f \in \mathbb{C}(X)^{\times}$. 
If $div(f)  = \sum_{i} m_{i} D_{i}$ then the vertical map on the left hand side is 
\[N\otimes f \mapsto \sum_{i} m_{i}N|_{D_{i}}
\]
\end{lemma}
{\bf Proof.}
This follows from the commutativity of the diagram below
\[
\xymatrix{\ar @{} [dr] |{}    & &\\
\mathbb{Z}_{X}
\ar[r] \ar[d]&  \mathbb{C}(X)^{\times} \ar[d] \\
\bigoplus_{D \in Supp(div(f))}\iota_{D*}\mathbb{Z}_{D} \ar[r] & \mathcal{M}^{\times}/\mathcal{O}^{\times}  }
\]
where the top arrow is $n \mapsto f^{n}$ and where the bottom arrow comes from the map $s$ defined in equation (\ref{eqn:Splitting}).

 \hfill $\Box$

Therefore given locally irreducible divisors $D_{i}$ and classes $n_{i} \in H^{p}(D_{i},\mathbb{Z})$.  The tensor product $\otimes_{i}B(n_{i})^{m_{i}}$ comes from a rational function $f$ if $div(f) = \sum_{i} m_{i} D_{i}$.  In the case where $X$ is non-compact, one can replace $\mathbb{C}(X)^{\times}$ with $\mathcal{M}^{\times}(X)$ in the above discussion.  Generic values of $B$ do not come from global functions.
\subsection{Topological Classes of Higher Gerbes coming from Divisors} \label{Top}
We can explain in purely topological terms the map 
\[H^{p}(\tilde{D},\mathbb{Z}) \to H^{p+1}(X,\mathbb{Z})
\]
which we have for any divisor $D$ of $X$.  It is a composition of three maps 
\[H^{p}(\tilde{D},\mathbb{Z}) \stackrel{B}\to H^{p}(X,\mathcal{M}^{\times}/\mathcal{O}^{\times}) \to H^{p+1}(X,\mathcal{O}^{\times}) \to H^{p+2}(X,\mathbb{Z})
\]
where $B$ was defined in (\ref{eqn:Bdef}) and the second and third maps are connecting maps from the long exact sequences associated to (\ref{eqn:HolMer}) and (\ref{eqx1}).  It is shown in \cite{Br1994} that the combined map is the pushforward (or transfer) map on integral cohomology.
\section{Producing Divisors With Interesting Topology}\label{oo1}
In this section we produce lower bounds on the Betti numbers of the normalizations of divisors.  Note that if we are interested in the non-triviality of (for instance)  $H^{1}(X,\mathcal{M}^{\times})$ there is a simple criterion for this in the case where 
\[H^2(X,\mathcal {O})=0.\]   In these cases, by Lemma \ref{splitting} and Remark \ref{x1} it is sufficient to
find an integral divisor $D \subset X$ whose normalization $\tilde{D}$ satisfies that the rank of $H^1(\tilde{D},\mathbb {Z})$ with is larger than 
the rank of
$H^3(X,\mathbb {Z}) \cong H^2(X,\mathcal {O}^{\times})$.  In section \ref{Local} we look at the corresponding problem at a boundary point, i.e. we take
$X$ as a proper open subset of another complex manifold $Y$, take $P\in \overline{X}\setminus X$. 
\begin{remark}\label{x1}
Fix an integer $i \geq 1$ and assume that $H^{i}(X,\mathcal{M}^{\times}/\mathcal{O}^{\times})$ is non-trivial.  If we know that $H^{i+1}(X,\mathcal{O}^{\times})$ is trivial then can conclude that $H^{i}(X,\mathcal{M}^{\times})$ is non-trivial.  More generally it is sufficient to find a subgroup of $H^i(X,\mathcal{M}^{\times}/\mathcal{O}^{\times})$ strictly larger that the abelian group $H^{i+1}(X,\mathcal{O}^{\times})$.
Conversely, if we know that the natural map $H^i(X,\mathcal {O}^{\times }) \to H^i(X,\mathcal {M}^{\times })$ is not surjective, then $H^i(X,\mathcal {M}^{\times }/\mathcal {O}^{\times })$ is non-trivial.  
\end{remark}
In sections \ref{oo1} and \ref{oo2} we will use the second part of Remark \ref{x1} to show that $H^i(X,\mathcal {M}^{\times }/\mathcal {O}^{\times })$ is non-trivial.  In those sections we will show that for many $X$ the abelian group $H^i(X,\mathcal {M}^{\times })$ is too large
to be a quotient of the abelian group $H^i(X,\mathcal {O}^{\times })$.
\begin{lemma}\label{u1}
Let $X$ be a smooth and connected quasi-projective complex surface. Fix an integer $r>0$. Then there exists a closed and smooth divisor $D\subset X$
such that $H^1(D,\mathbb {Z})$ has $\mathbb {Z}^r$ as a direct factor.
\end{lemma}

{\bf Proof.}

Fix any effective divisor $D\subset X$. Since $D$ is quasi-projective, all cohomology groups $H^i(D,\mathbb {Z})$ are finitely generated. The universal coefficient theorem shows that $H^1(D,\mathbb {Z})$ has no torsion. Hence
it is sufficient to prove the existence of an effective divisor $D$ such that $\mathbb {Z}^r \subseteq  H^1(D,\mathbb {Z})$. Since $X$ is a quasi-projective and smooth surface, there is an open embedding $X\hookrightarrow  Y$ with $Y$ a smooth and connected projective surface.
Fix a very ample divisor $H$ on $Y$. Set $c= H \cdot H$ (self-intersection) and $e= K_X \cdot H$.  Since $H$ is ample, $c$ is a positive integer. For every integer $k>0$ the divisor $kH$ is very ample. Hence a general $D_k\in \vert kH \vert$
is a smooth and connected curve. Call $g_k$ its genus. The adjunction formula gives $k^2c + ke = 2g_k-2$. Thus for $k\gg 0$ we have $2g_k \ge r$. Fix any such integer $k$
and set $D=D_k\cap X$. Since $D$ is the complement in $D_k$ of finitely many points and $H^1(D_k,\mathbb {Z}) \cong \mathbb {Z}^{2g_k}$, we are done.

 \hfill $\Box$

\begin{corollary}\label{u2}
Let $X$ be a smooth and connected quasi-projective surface such that $H^1(X,\mathcal {O}) =0$. Then for every integer $t>0$ the abelian group $H^{1}(X,\mathcal{M}^{\times})/H^{1}(X,\mathcal{O}^{\times})$
contains a subgroup isomorphic to $\mathbb {Z}^t$. 
\end{corollary}

{\bf Proof.}

Look at the exponential sequence of sheaves on $X$ given by
\begin{equation}\label{eqx1}
0 \to \mathbb {Z} \to \mathcal {O} \to \mathcal {O}^\times \to 1.
\end{equation}
From (\ref{eqx1}) we get $H^1(X,\mathcal {O}^{\times })\subseteq H^2(X,\mathbb {Z})$. Thus $H^1(X,\mathcal {O}^{\times })$
is a finitely generated abelian group. Apply
Lemma \ref{u1} and Remark \ref{x1}. 







 \hfill $\Box$

Let $X$ be a smooth and connected projective surface, i.e. take the set-up of Corollary \ref{u2} with $X$ is compact. The group $H^2(X,\mathcal {O})$ is a finite-dimensional
$\mathbb {C}$-vector space and its dimension is usually denoted with $p_g(X)$. We have $p_g(X)=0$ if $X$ is a ruled surface and in particular if $X$ is a rational surface. There are
other scattered examples with $p_g(X)=0$, e.g. the Enriques surfaces or some surfaces of general type with $q(X)=h^{1}(X,\mathcal{O})=0$ (for surfaces of general type
$p_g(X)=0$ implies $q(X)=0$, because $\chi (\mathcal {O}_X)\ge 0$ for every surface $X$ of general type). The rational surfaces and the blowing-ups of Enriques surfaces
satisfies $H^1(X,\mathcal {O}) =H^2(X,\mathcal {O}) =0$.

\begin{thm}\label{n1}
Let $X$ be an $n$-dimensional connected projective manifold, $n \ge 2$. Then there exists an integral closed
divisor $D\subset X$ such that the integral cohomology of its normalization $\tilde{D}$ contains (up to torsion) a subring given by the cohomology of any smooth projective variety $A$ of dimension $n-1$.  To be more specific
\[H^{*}(A,\mathbb{Z})/\mbox{Tors}(H^{*}(A,\mathbb {Z})) \subset H^{*}(\tilde{D},\mathbb{Z})/\mbox{Tors}(H^{*}(\tilde{D},\mathbb {Z})).\]  By taking $A$ to be a product of a curve with itself $n-1$ times we get that for every $t>0$ one can find a divisor $D$ such that 
\[\mathbb{Z}^{t} \subset H^{i}(\tilde{D},\mathbb{Z})\]
for all integers $i$ such that $1 \leq i \leq 2n-3.$
\end{thm}

{\bf Proof.}

If $n=2$, then the result is true by Lemma \ref{u1} even if $X$ is only assumed to be quasi-projective (and with $D = \tilde{D}$ smooth). Assume
$n\ge 3$. Let $A$ be an $(n-1)$-dimensional smooth projective variety. Fix any embedding \[j: A\hookrightarrow  \mathbb {P}^r,\] where $r \ge 2n-1$ and take a general
projection of $j(A)$ into $\mathbb {P}^n$ from a general $(n-r-1)$, dimensional linear subspace $W\subset \mathbb {P}^r$. Since $\dim (A)+\dim (W) < r$, for general
$W$ we have $W\cap j(A) =\emptyset$. Thus the linear projection induces a morphism
\[u: A \to \mathbb {P}^n.\] Since $W\cap j(A)=\emptyset$, the morphism $u$ is finite. Since $\dim (A)<n$ and $W$ is general, $u$ is birational onto its image (\cite{h}, Ex. I.4.9, page 31).  Hence because $A$ is normal (in fact smooth) and $u$ is finite and birational onto its image, we can conclude that \[u: A \to u(A)\] is the normalization
map (\cite{Na1966}, Definition 2 at page 114).
Fix a finite and dominant morphism \[f: X \to \mathbb {P}^n.\]  Let $D\subset X$ be any irreducible $(n-1)$-dimensional
component
of $f^{-1}(u(A))$. Thus $f(D) = u(A)$. Let 
\[\mu : \tilde{D} \to D\] be the normalization map. The morphism \[f\circ \mu : \tilde{D} \to u(A)\] is a finite
dominant morphism between integral varieties. Since $A$ is the normalization of $u(A)$ the universal property of normal varieties (\cite{h}, Ex. I.3.17, or \cite{s}, Theorem 5 at p. 115, or \cite{f}, Proposition at p. 121)
implies the existence of a morphism \[g: \tilde{D} \to A\] such that \[f\circ \mu = u\circ g.\]  

The map $g$ is finite and has some topological degree $m>0$.  Let \[r:\tilde{D}' \to \tilde{D}\] be a resolution of singularities of $\tilde{D}$ which is known to exist by Hironaka's Theorem \cite{Hiro}.  The morphism $r$ has degree $1$ and $\tilde{D}'$ is a smooth projective manifold of dimension $n-1$.   Now
we use the pushforward (or transfer) map on cohomology  (see for example Definition VIII.10.5 on page 310 of \cite{Dold}). For any class $c \in H^{p}(A,\mathbb{Z})$ the projection formula \cite{Dold} gives 
\[{(g \circ r)}_{*} {(g \circ r)}^{*} c = c \cup {(g \circ r)}_{*} 1 = mc.
\]
Therefore the kernel of $(g \circ r)^{*}$ is contained in the kernel of multiplication by $m$.  Hence the ring homomorphism $r^{*} \circ g^{*} =(g\circ r)^{*}$ is injective, up to torsion.  Therefore $g^{*}$ in an injective homomorphism of rings, up to torsion.  If one choses $A = C^{n-1}$ for $C$ a curve of genus $g \geq \max \{2,t\}$ then $H^{i}(A,\mathbb{Z})$ has no torsion and from the proof we actually see that $g^{*}$ is injective.  The K\"{u}nneth formula therefore tells us that $\mathbb{Z}^{t}$ is contained in $H^{i}(\tilde{D},\mathbb{Z})$ 
for $1 \leq i \leq 2n-3$. 

 \hfill $\Box$

As in Corollary \ref{u2} from Theorem \ref{n1} (case $i\le 2n-3$) and Proposition \ref{ee} (case $i=2n-2$) we get the following result.

\begin{corollary}\label{n2}
Fix an integer $t>0$. Let $X$ be an $n$-dimensional connected projective manifold, $n \ge 2$ such that $H^{i}(X,\mathcal{O}) = 0$.  Then $H^{i}(X,\mathcal{M}^{\times})/H^{i}(X,\mathcal{O}^{\times})$
contains a subgroup isomorphic to $\mathbb {Z}^{t}$.
\end{corollary}

\begin{proof}
Since $H^i(X,\mathcal {O})=0$, the exponential sequence shows that $H^i(X,\mathcal {O}^{\times })$ is a subgroup of $H^{i+1}
(X,\mathbb {Z})$. Hence $H^i(X,\mathcal {O}^{\times })$ is finitely generated. Call $\rho$ its rank. Apply Theorem \ref{n1} (case $i\le 2n-3$) and Proposition \ref{ee} (case $i=2n-2$)
with respect to the integer $t'= \rho +t$.
\end{proof}

\section{Consequences and Examples}
\subsection{The Question of Chen, Kerr, and Lewis}\label{Question}
We can now answer a question posed by Chen, Kerr, and Lewis in \cite{ChKeLe10}.  They asked if on a smooth projective complex algebraic variety $X$, if the sheaves $\mathcal{M}^{\times}/\mathbb{C}(X)^{\times}$ are acyclic.   The answer is negative.  For convenience we show that the answer is negative in a two specific examples although it seems to be a fairly general phenomenon.  We will use the short exact sequence 
\begin{equation} \label{eqn:AlgMer}
1 \to \mathbb{C}(X)^{\times} \to \mathcal{M}^{\times} \to \mathcal{M}^{\times}/ \mathbb{C}(X)^{\times} \to 1
\end{equation}
 For $X = \mathbb{P}^{2}$ over the complex numbers, $H^{1}(X, \mathcal{M}^{\times}/\mathbb{C}(X)^{\times})$ and $H^{2}(X, \mathcal{M}^{\times}/\mathbb{C}(X)^{\times})$ cannot vanish.  
The first example is easier.  Indeed,  $H^{2}(\mathbb{P}^{2},\mathcal{O}^{\times})$ is trivial as can be seen from the exponential sequence (\ref{eqx1}).  Therefore a quotient group of $H^{1}(\mathbb{P}^{2},\mathcal{M}^{\times})$ is given by $H^{1}(\mathbb{P}^{2},\mathcal{M}^{\times}/\mathcal{O}^{\times})$.  This is a huge group as shown in Lemma \ref{u1}.  It contains for instance the first integral cohomology of an elliptic curve sitting inside $\mathbb{P}^{2}$.  Therefore $H^{1}(\mathbb{P}^{2},\mathcal{M}^{\times})$ cannot be trivial.  Of course, by change of coefficient groups and the fact that $H^{1}(\mathbb{P}^{2},\mathbb{Z})= 0$ we have that $H^{1}(\mathbb{P}^{2},\mathbb{C}(\mathbb{P}^{2})^{\times})$ is trivial. Therefore the long exact sequence associated to the exact sequence (\ref{eqn:AlgMer}) shows that $H^{1}(\mathbb{P}^{2}, \mathcal{M}^{\times}/\mathbb{C}(\mathbb{P}^{2})^{\times})$ contains $H^{1}(\mathbb{P}^{2},\mathcal{M}^{\times})$ and so cannot be trivial.

Now we prove the non-triviality of $H^{2}(\mathbb{P}^{2},\mathcal{M}^{\times}/\mathbb{C}(\mathbb{P}^{2})^{\times})$. Look at the long exact sequence associated to the exponential sequence (\ref{eqx1}). It shows that \[H^{3}(\mathbb{P}^{2},\mathcal{O}^{\times}) \cong H^{4}(\mathbb{P}^{2},\mathbb{Z}) = \mathbb{Z}[\mathbb{P}^2].\]  Therefore the long exact sequence associated to  (\ref{eqn:AlgMer}) reads 
\[\{1\} = H^{2}(\mathbb{P}^{2},\mathcal{O}^{\times}) \to H^{2}(\mathbb{P}^{2},\mathcal{M}^{\times}) \to H^{2}(\mathbb{P}^{2},\mathcal{M}^{\times}/\mathcal{O}^{\times}) \to H^{4}(\mathbb{P}^{2},\mathbb{Z}) =\mathbb{Z}[\mathbb{P}^2].
\]On the other hand by change of coefficient groups  $H^{2}(\mathbb{P}^{2},\mathbb{C}(\mathbb{P}^{2})^{\times}) = \mathbb{C}(\mathbb{P}^{2})^{\times}$.  Let $C_{1}$ and $C_{2}$ be two smooth curves of degrees $d_{1}$ and $d_{2}$ in $\mathbb{P}^{2}$ such that $d_{1} \neq d_{2}$.  They have fundamental classes $[C_{1}] \in H^{2}(C_{1},\mathbb{Z})$ and $[C_{2}] \in H^{2}(C_{2},\mathbb{Z})$.  The pushforward of either class to $H^{4}(\mathbb{P}^{2},\mathbb{Z})$ is the generator, $[\mathbb{P}^{2}]$ because this corresponds via Poincare Duality to pushing forward a point in the zeroth homology group.  Therefore using subsection \ref{Top} and the map $B$ defined in equation  (\ref{eqn:Bdef}) we can consider the class \[B([C_{1}]) \otimes B([C_{2}])^{-1} \in H^{2}(\mathbb{P}^{2},\mathcal{M}^{\times}/\mathcal{O}^{\times}).\]   This maps to the trivial class in  $H^{4}(\mathbb{P}^{2},\mathbb{Z})$  and therefore comes from some class  \[\widetilde{B([C_{1}]) \otimes B([C_{2}])^{-1}} \in H^{2}(\mathbb{P}^{2},\mathcal{M}^{\times}).\]  Suppose that the natural map \[\mathbb{C}(\mathbb{P}^{2})^{\times} = H^{2}(\mathbb{P}^{2},\mathbb{C}(\mathbb{P}^{2})^{\times}) \to H^{2}(\mathbb{P}^{2},\mathcal{M}^{\times})\] were surjective and consider a lift to $H^{2}(\mathbb{P}^{2},\mathbb{C}(\mathbb{P}^{2})^{\times})$ of the class $\widetilde{B([C_{1}]) \otimes B([C_{2}])^{-1}}$.  We can write this lift as the cup product of an element $N \in H^{2}(\mathbb{P}^{2},\mathbb{Z}) = \mathbb{Z}$ and a function $f \in \mathbb{C}(\mathbb{P}^{2})^{\times}$.  We denote this cup product as the element $f^{N} \in H^{2}(\mathbb{P}^{2},\mathbb{C}(\mathbb{P}^{2})^{\times})$.  This immediately means that we must have integers $e_{1}$ and $e_{2}$ such that 
\[div(f) = e_{1} C_{1} + e_{2} C_{2}
\]
where $e_{1} d_{1} + e_{2} d_{2} = 0$.  Then the divisor of $f^{N}$ and the divisor of $B([C_{1}]) \otimes B([C_{2}])^{-1}$ must agree as elements of $\bigoplus_{D \in IDiv(\mathbb{P}^{2})} H^{2}(D,\mathbb{Z})$.    Thus the divisor of $f^{N}$ is \[
(Ne_{1}d_{1},Ne_{2}d_{2}) \in H^{2}(C_{1},\mathbb{Z}) \oplus H^{2}(C_{2},\mathbb{Z})
\] because the natural restriction maps $H^{2}(\mathbb{P}^{2},\mathbb{Z}) \to H^{2}(C_{i},\mathbb{Z})$ are multiplication by $d_{i}$.  However the divisor of $B([C_{1}]) \otimes B([C_{2}])^{-1}$ is just \[(1,-1) \in H^{2}(C_{1},\mathbb{Z}) \oplus H^{2}(C_{2},\mathbb{Z})
\] 
therefore the agreement of these divisors implies $d_{1}, d_{2} =  1$ and so the degrees need to agree which contradicts our assumption.

\subsection{The Examples of $\mathbb{C}^{2}$ and $\mathbb{P}^{2}$}\label{c2}
\begin{ex}
As we will now see, even $\mathbb{C}^{2}$ admits non-trivial meromorphic line bundles!   In more generality we observe that any complex manifold $X$ admits non-trivial meromorphic line bundles as long as $H^{2}(X,\mathcal{O}^{\times})$ is trivial and $X$ has an irreducible  divisor $D$ with $H^{1}(D, \mathbb{Z}) \neq 0$.  Another simple example of this would be $\mathbb{P}^{2}_{\mathbb{C}}$ where the divisor is a smooth curve of genus greater than or equal to one, or any curve with non-zero first Betti number.    Let $\mathbb{P}^{1}_{\mathbb{C}} \cong Q \subset \mathbb{P}^{2}_{\mathbb{C}}$ be a smooth quadratic curve, meeting some divisor at infinity at two points.  Let $V = \mathbb{C}^{2}$ be the compliment of the divisor at infinity and $D = Q \cap V \cong \mathbb{C}^{\times}$.  Let $n \in H^{1}(V,\mathbb{Z}_{D}) = H^{1}(D, \mathbb{Z})$ be a non-trivial element.  Then by the above argument, we can consider $n$ as a non-trivial element of $H^{1}(\mathbb{C}^{2}, \mathcal{M}^{\times}/\mathcal{O}^{\times})$.  Because all gerbes are trivial in this situation, we see that we can lift $n$ to a non-trivial element of $H^{1}(\mathbb{C}^{2}, \mathcal{M}^{\times})$.  In fact for any $X$ as above
\[\bigoplus_{D \in IDiv(X)} H^{1}(D, \mathbb{Z})  \subset H^{1}(X,\mathcal{M}^{\times})/H^{1}(X,\mathcal{O}^{\times}).
\]
\end{ex}
The cases of $\mathbb {C}^2$ and $\mathbb {P}^2$ also follows from Remark \ref{x1} and Corollary \ref{u2}.

\subsection{Complex Tori}
A {\it generic} complex torus $X$ has no global divisors and so sequence (\ref{eqn:HolMer}) shows that we have $Pic^{0}(X) \subset H^{1}(X,\mathcal{M}^{\times})$.    On the other hand if a complex torus is algebraic, then all holomorphic line bundles are meromorphically trivializable.  Holomorphic gerbes on complex tori were studied extensively in \cite{Be2009}.  Even when the Neron-Severi group of a complex torus is zero, it was shown that this complex torus can still have non-trivial holomorphic gerbes hence the correspondence we have studied in this paper producing classes in $H^{2}(X,\mathcal{O}^{\times})$ from divisors cannot be surjective in general.

\section{Local Analysis}\label{Local}

\subsection{The cohomology of $\mathcal{M}^{\times}$ near boundary points} \label{oo2}
Let $A[X]$ be the group defined by  \[A[X] = \text{im}[H^1(X,\mathcal {M}^\times ) \to H^1(X,M^{\times}/\mathcal{O}^{\times})].\] Let $X\subsetneqq Y$ be a connected open subset of a smooth and connected complex surface $Y$. Fix $P\in \overline{X}\setminus X$. Let $\mathcal {V}_P$ denote
the filter of all open neighborhoods of $P\in Y$. Fix any abelian sheaf $F$ on $X$. For each integer $i\ge 0$ let $H^i(X,F)_P$ be defined by 
\[H^{i}(X,F)_P =\lim _{U\in \mathcal {V}_P} H^i(X\cap U,F).\] Here we are interested in the case $i=1$ and either $F = \mathcal {M}^\times $ or $F = \mathcal{M}^{\times}/\mathcal{O}^{\times}$. We also want
to study the image of the natural map 
\[A_{X,P}=\text{im}[H^1(X,\mathcal {M}^\times )_P \to H^1(X,\mathcal{M}^{\times}/\mathcal{O}^{\times})_P].\] We consider
a class of boundary points $P$ for which $A_{X,P}$ has $\mathbb {Z}$ as a direct factor and in particular it is non-zero. We say that $X$ is {\it pseudoconcave}
at $P$ if there is $U\in \mathcal {V}_P$ and a closed analytic subset $D$ of $U$ such that $D$ is biholomorphic to a disk, $P\in D$ and $D\cap X = D\setminus \{P\}$.
This notion is inspired to the usual notion of pseudoconcave manifold (\cite{a1}, \cite{a2}, 2.3).
\begin{proposition}\label{i1}
Assume that $X$ is pseudoconcave at $P$. Then $A_{X,P}$ has $\mathbb {Z}$ as a direct factor.
\end{proposition}

{\bf Proof.}

Fix $U\in \mathcal {V}_P$ such that there is a closed analytic subset $D\subset U$ with $D$ a disc and $D\cap X$ a puntured disk. Thus $H^1(D,\mathbb {Z}) \cong \mathbb {Z}$.
Lemma \ref{splitting} gives that $A[X\cap U]$ contains $\mathbb {Z}$ as a direct factor. Moreover, if we take a smaller neighborhood $V$ of $X$, this factor will map isomorphically
onto a factor of $A[X\cap V]$, because the map $H^1(D,\mathbb {Z}) \to H^1(D\cap V,\mathbb {Z})$ is injective (the loop around $P$ inducing a generator
of $H^1(D,\mathbb {Z})$ gives a non-torsion element of $H^1(D\cap V,\mathbb {Z})$).

 \hfill $\Box$
 
 Now we describe a cheap way to produce complex manifolds $X\subseteqq Y$ of dimension $n \ge 2$ and $P\in \overline{X}\setminus X$ such that $X$ is pseudoconcave at $P$.
Let $Y$ be a connected complex manifold. Fix $Q\in Y$ and and open neighborhood $U$ of $Q$ in $Y$ equipped with a biholomorphic map $j: U \to j(U)$
with $j(U)$ open subset of $\mathbb {C}^n$. Fix an open and strictly convex neighborhood $D$ of $j(Q)$ in $\mathbb {C}^n$ with $D\subset j(U)$ .
Set $X:= j^{-1}(\overline{D})$. Thus $\overline{X} = Y\setminus j^{-1}(D)$. Fix any $P\in \overline{X}\setminus X$, i.e. any $j(P)\in \partial D$. Since $\overline{D}$ is strictly
convex at $P$, there is a real affine hyperplane $H$ of $\mathbb {C}^n$ such that $H\cap \overline{D} =\{j(P)\}$. Since $n \ge 2$, there is an affine complex line $L\subset H$
such that $P\in L$. Hence there is an open disk $A$ of $\mathbb {C}$ such that $\overline{A} \subset j(U)$. The open subset $j^{-1}(A)$ of $Y$ is biholomorphic to
an open disc and $j^{-1}(A) \cap X = j^{-1}(A)\setminus \{P\}$. Thus $X$ is pseudoconcave at $P$.

\subsection{The cohomology of $\mathcal{M}^{\times}$ on non-compact complex manifolds}
In this subsection we give some results and examples on the size of $H^{i}(X,\mathcal{M}^{\times})$ for non-compact spaces.  The methods are similar to those used in  \cite{ChKeLe10}. 

\begin{prop}\label{f1}
Fix an integer $i \ge 1$. Let $X$ be a connected and paracompact $n$-dimensional complex manifold, $n \ge 2$, and $f$ a non-constant holomorphic function on $X$. For each $t\in \mathbb {C}$ set 
\[D_t= \{x |f(x)=t\}
\] with its scheme structure. Fix $\mu \in H^i(X,\mathbb {Z})$. For each $t\in \mathbb {C}$ such
that $D_t$ has no multiple component let 
\[\rho _t:  H^i(X,\mathbb {Z}) \to H^i(D_t,\mathbb {Z})\] denote the restriction map. Set $G_t:= Im (\rho _t)$.

\quad (i) Fix any $t\in \mathbb {C}$ such that $D_t$ is non-empty, reduced, irreducible and with singular locus of dimension at most $n-3$. Then there is an inclusion
\[j_t: G_t \hookrightarrow H^i(X,\mathcal {M}^\times)\] of abelian groups.

\quad (ii) Fix a set $I\subseteq \mathbb {C}$ such that $D_t$ is reduced, irreducible and with singular locus of dimension at most $n-3$ for each $t\in I$.
Assume that $\mu
_t:=
\rho _t(\mu )$ is not torsion for each $t\in I$ (case $i\ge 2$) or $\mu _t \ne 0$ for each $t\in I$ (case $i=1$). For each $t\in I$ set $u_t:= j_t(\mu _t)$. Then the elements $u_t$ in
$H^i(X,\mathcal {M}^\times )$ obtained using all $f_t$'s are $\mathbb {Z}$-independent. Hence $H^i(X,\mathcal {M}^\times )$
contains a free abelian group with a basis with the same cardinality as $I$.
\end{prop}

{\bf Proof.}

Fix $t\in \mathbb {C}$ such that $D_t$ is non-empty. Since $X$ is smooth, each local ring $\mathcal {O}_{X,x}$, $x\in X$, has depth $n$. Hence
for each $t\in \mathbb {C}$ and each $x\in D_t$, the local ring $\mathcal {O}_{X_t,x}$ has depth $n-1$. Thus if the singular locus of $D_t$ has codimension
at least $2$ in $D_t$, then $D_t$ is a normal complex space (\cite{ei} Theorem 1.5 or Exercise 11.10). Hence for every 
non-empty and connected
open subset $U$ of $D_t$ the manifold
$U_{reg}=U\cap (D_t)_{reg}$ is connected.  The assumptions on the singular locus of $D_{t}$ imply that $D_{t}$ is locally irreducible and hence one can calculate the divisor of any meromorphic function defined in a small enough neighborhood of any point of $D_{t}$.

\quad (a) Fix $t\in I$. To get part (i) it is sufficient to prove $j_t(\rho _t(\eta ))\ne 0$ for each $\eta \in H^i(X,\mathbb {Z})$ such that $\rho _t(\eta )\ne 0$. Let $\{V_\alpha \}$ be an open covering
of $X$ such that all finite intersections of elements of $V_\alpha$ and all finite intersections of elements of $V_\alpha$ and $D_t$ are either empty or contractible (and in particular connected). It exists, because we may take a triangulation of $X$ for which $D$ is union of cells.
Notice that $\{V_\alpha \}$ is an acyclic covering of $X$ for the sheaf
$\mathbb {Z}$
such that $\{V_\alpha \cap D_t\}$ is an acyclic covering of $D_t$ for the sheaf $\mathbb {Z}$.
Thus the \v{C}ech cohomology of $\{V_\alpha \}$ with coefficients in $\mathbb{Z}$ computes $H^i(X,\mathbb {Z})$ and the \v{C}ech cohomology of $\{V_\alpha \cap D_t \}$  with coefficients in $\mathbb{Z}$ computes $H^i( D_t ,\mathbb {Z})$. Let 
$m =\{m_{\alpha _0\cdots \alpha _i} \}$
be a cocycle computing $\eta$ with respect to the cover $\{V_\alpha \}$. The maps $j_{t}$ is defined by letting $j_t(\rho _t(\eta ))$ to be represented by the $i$-cocycle $g$ with coefficients of $\mathcal {M}_X^\times$ with respect to  $\{V_\alpha \}$ where 
\[g_{\alpha _0\cdots \alpha _i}=
(f-t)^{m_{\alpha _0\cdots \alpha _i}}.\]  Since $f-t$ is fixed, the map $j_t$ is a homomorphism of abelian groups. We claim that this cocycle is not cohomologous to zero. Indeed, assume that it is cohomologous to zero. The map
$j_t$ commutes with refinements of open coverings of $X$. Thus (refining if necessary the covering) we may assume the existence of an $(i-1)$-cochain $h = \{h_{\alpha _0\cdots \alpha _{i-1}} \}$ with respect to  $\{V_\alpha \}$
such that $g = \delta (h)$, where $\delta$ is the \v{C}ech boundary. Let $e_{\alpha _0\cdots \alpha _{i-1}}$ be the multiplicity of $h_{\alpha _0\cdots \alpha _{i-1}}$
along $D_t\cap V_{\alpha _0\cdots \alpha _{i-1}}$; $e_{\alpha _0\cdots \alpha _{i-1}}$ is a well-defined integer, because $(D_t)_{reg}\cap V_{\alpha _0\cdots \alpha _{i-1}}$ is connected. Looking at the
order of vanishing along $D_t$ we see that if $V_{\alpha _0\cdots \alpha _i}\cap D_t\ne \emptyset$, then 
\[m_{\alpha _0\cdots \alpha _i} = (\delta e)_{\alpha _0\cdots \alpha _i}.\] Thus the restriction $m|_{D_{t}}$ is a trivial cocycle with coefficients in $\mathbb{Z}$ with respect to $\{V_\alpha \cap D_t\}$, a contradiction.

\quad (b) Fix $I\subseteq \mathbb {C}$, $\mu$ and $u_t$ as in part (ii). Let $b$ be a cocycle computing $\mu$ with respect to the cover $\{V_{\alpha}\}$.   
 Suppose that the elements $u_t$, $t\in I$, are not $\mathbb {Z}$-independent. Hence there are finitely many $t_j\in I$, $1\le j \le s$, and integers $n_j$, $1 \le j
\le s$, such that the cochain $g'$ defined by
\[g'_{\alpha _0\cdots \alpha _i} = \prod _{j=1}^{s}(f-t_j)^{n_jm_{\alpha _0\cdots \alpha _i}}\] is cohomologous to zero and at least one of the $n_{j}$ is not zero. Thus (refining if necessary the covering) we may assume the existence of an $(i-1)$-cochain $h'$ with respect to  $\{V_\alpha \}$
such that $g' = \delta h'$. We refine the covering $\{V_\alpha \}$ so that finite intersections of $V_\alpha$ and $D_{t_j}$, $1 \le j \le s$,
are either empty or connected and that $V_\alpha \cap D_x \cap D_y =\emptyset$ for all $\alpha$ and all $x, y\in \{t_1,\dots ,t_s\}$ such that $x\ne y$. Fix $j\in \{1,\dots ,s\}$. For each open set $V_{\alpha _0\cdots \alpha _i}$ such that $V_{\alpha _0\cdots \alpha _i}\cap D_{t_j} \ne \emptyset$ 
the function $g'_{\alpha _0\cdots \alpha _i}$ has order of vanishing (or pole) $n_jm_{\alpha _0\cdots \alpha _i}$ at a general
point of $V_{\alpha _0\cdots \alpha _i}\cap D_{t_j}$. Call $c_{\alpha _0\cdots \alpha _{i-1}}$ the order of vanishing or pole
of $h'_{\alpha _0\cdots \alpha _{i-1}}$ at a general point of $V_{\alpha _0\cdots \alpha _{i-1}}\cap D_{t_j}$ with
the convention that this order is zero if $V_{\alpha _0\cdots \alpha _{i-1}}\cap D_{t_j} = \emptyset$. We get a $i-1$ cochain
$c
= \{c_{\alpha _0\cdots \alpha _{i-1}}\}
$ for the sheaf $\mathbb{Z}$ with respect to the cover $\{ V_{\alpha} \cap D_{t_j} \}$ such that \[\delta (c)=n_j\rho_{t_j}(b)\]
and therefore $n_{j} \rho_{t_j}(\mu) = 0$.
The universal coefficient theorems gives that $H^1(D_{t_j},\mathbb {Z})$ has no torsion. Hence $\mu _{t_j}$ is not torsion, even in the case $i=1$. Since $\rho _{t_j}(\mu )$ is not a torsion class, we get $n_j=0$. Since this is true for all $j\in \{1,\dots ,s\}$,
we get a contradiction and prove part (ii).

 \hfill $\Box$

In some particular cases we may find a set $I\subseteq \mathbb {C}$ as in part (b) of Proposition \ref{f1} and with cardinality $2^{\aleph _0}$ (even when $H^i(X,\mathbb {Z})$ is a finitely generated
abelian group). We give the following example.

\begin{ex}\label{f2}
Fix an $(n-1)$-dimensional connected Stein manifold $Y$
such that $H^1(Y,\mathbb {Z}) \ne 0$, an open subset $\Omega \subseteq \mathbb {C}$ and $a\in \Omega$.
Set $X:= \Omega \times Y$ and $f:= \pi _1$, where $\pi _1: X \to \Omega$ denotes the projection. We may take $I:= \Omega$. Hence $I$
has  cardinality $2^{\aleph _0}$. Since $Y$ and $\Omega$ are Stein spaces, $X$ is a Stein space. Hence $H^k(X,\mathcal {O})=0$ for all $k\ge 1$.
Thus the exponential sequence gives $H^k(X,\mathcal {O}^{\times}) \cong H^{k+1}(X,\mathbb {Z})$ for all $k\ge 1$. Since $\Omega$ is an open subset of $ \mathbb {C}$, we have
$H^m(\Omega ,\mathbb {Z})=0$ for all $m\ge 2$. Hence if $\Omega$ is simply connected and $H^2(Y,\mathbb {Z}) = H^3(Y,\mathbb {Z}) =0$, then K\"{u}nneth's formula gives
$H^2(X,\mathbb {Z}) = H^3(X,\mathbb {Z})=0$. Thus $H^1(X,\mathcal {O}^{\times })$ and $H^2(X,\mathcal {O}^{\times })$ are both trivial. Thus $H^1(X,\mathcal {M}^\times ) \cong H^1(X,\mathcal{M}^{\times}/\mathcal{O}^{\times})$.
\end{ex}

\section{Open Questions}
Are the functors $X \mapsto H^{1}(X,\mathcal{M}^{\times})$ representable and can interesting birational invariants be extracted from these groups?  Are there classes in $H^{1}(X,\mathcal{M}^{\times}/\mathcal{O}^{\times})$ that do not come from any divisor from via the map $B$?   This question concerns the question of the acyclicity of the intersection over all irreducible divisors $Y$ of the sheaf $\mathcal{C}_{Y}$ found in equation (\ref{CEqn}).  This question would be especially interesting on a complex manifold without global divisors.

\vskip 0.2in \noindent {\scriptsize {\bf Edoardo Ballico,}
Dept. of Mathematics, University of Trento, 38123 Povo (TN), Italy,
ballico@science.unitn.it }
\vskip 0.2in \noindent {\scriptsize {\bf Oren Ben-Bassat,}
Department of Mathematics, University of Haifa, Mount Carmel, Haifa 31905, Israel, oren.benbassat@gmail.com}


\end{document}